\documentclass[12pt,a4paper]{amsart}

\usepackage{amssymb}
\usepackage[final]{showkeys}
\usepackage{microtype}
\usepackage{color}
\usepackage{graphicx}
\usepackage[hmargin=3cm,vmargin={5cm,4cm}]{geometry}

\newtheorem{te}{Theorem}

\newtheorem{os}{Remark}
\newtheorem{prop}{Proposition}

\numberwithin{equation}{section}

\allowdisplaybreaks

\def\eg{{\it e.g.}\ }

\def\e{\hbox{e}}

\def\ds{\displaystyle}
\def\RR{\mathbb{R}}

\def\e{{\rm e}}

\def\ds{\displaystyle}

\def\rec#1{\frac{1}{#1}}

\def\RR{\vbox {\hbox to 8.9pt {I\hskip-2.1pt R\hfil}}}

\begin{document}

\title[Some applications of Wright functions ]{Some applications of Wright functions in fractional differential equations.}
\author{R.Garra$^1$}
\address{$^1$ Dipartimento di Scienze Statistiche, Sapienza University of Rome\\
P. le Aldo Moro, 5 – 00185, Rome, Italy\\
E-Mail Address: roberto.garra@uniroma1.it}
\author{F. Mainardi $^2$}
\address{$^2$ Dipartimento di Fisica e Astronomia (DIFA), University of Bologna
"Alma Mater Studiorum", and INFN,\\ Via Irnerio 46, I-40126 Bologna,
Italy\\
E-mail address: francesco.mainardi@bo.infn.it 
 }

\begin{abstract}
In this note we prove some new results about the application of Wright functions
of the first kind  to solve 
fractional differential equations with variable coefficients.
Then, we consider some applications of these results in order to obtain some new particular solutions for nonlinear fractional 
partial differential equations.

\medskip
   	
   
\end{abstract}
 
\maketitle
\centerline{ {\bf Reports on Mathematical Physics, Vol. 87 No 2 (2021), pp 265--273}}
\centerline{{\bf DOI:10.1016/S0034-4877(21)00029-X}}

\section{Introduction}

The problem to find explicit solutions for fractional ordinary differential equations with 
variable coefficients is a topic of interest, also in the context of the studies about 
special functions. For example the so called Saigo-Kilbas function emerges in the study of
fractional ODE with variable coefficient (see \cite{oli} and the references therein).
Recently different authors are working on systematic methods to solve fractional ODEs with variable coefficients, we refer in particular to the recent paper \cite{ang}.
In other cases, it is possible to prove that some interesting classes of special functions 
solve fractional equations with variable coefficients. For example, in \cite{fede} the authors 
proved that a generalized Le Roy function solves a particular integro-differential equation with variable coefficients involving Hadamard fractional operators.\\
In the recent paper \cite{fil}, a new interesting result has been pointed out about the solution of a new class of fractional ODEs by means of the classical Wright functions of the first kind.

In particular, it was proved, by direct calculations, that the following
fractional ODE
\begin{equation}
\frac{d}{dt}t^\lambda\frac{d^\lambda u}{dt^\lambda}-\frac{\lambda u}{t^{1-\lambda}} = 0, \quad t\geq 0, \lambda \in (0,1),
\label{eq:1}
\end{equation}
involving the Caputo fractional derivative of order $\lambda$ is solved by the function
\begin{equation}
u(t) = W_{\lambda, 1}(t^\lambda) = \sum_{k=0}^\infty \frac{t^{\lambda k}}{k! \Gamma(k\lambda +1)}, \quad \lambda \in(0,1),
\label{eq:W1}
\end{equation}
under the initial condition that $u(t = 0) = 1$.
We note that  $W_{\lambda,1}$ is a particular case of 
 special transcendental functions known as  Wright functions that we will briefly discuss in the next section, distinguishing them in two kinds according to the values of the first parameter $\lambda$.   
We recall that the fractional derivative in the sense of Caputo of order $\nu>0$ is defined as
\begin{equation}
\left(\frac{d^\nu}{dt^\nu}f\right)(t):= 
\begin{cases}
&\displaystyle\frac{1}{\Gamma(m-\nu)}\int_0^t \frac{f^{(m)}(\tau)}{(t-\tau)^{\nu+1-m}}d\tau, \quad \mbox{for $m-1<\nu< m$},\\
& \displaystyle\frac{d^m}{dt^m} f(t), \quad \mbox{for $\nu = m$}.
\end{cases}
\end{equation}

Therefore, Eq. (\ref{eq:1})  can be viewed as a sort of fractional generalization of the Bessel-type differential equations.
Indeed, for $\lambda = 1$ we obtain the following equation
\begin{equation}\label{lag}
\frac{d}{dt}t \frac{d u}{dt} = u,
\end{equation}
whose solution is the  so-called Tricomi function (see \cite{datt}, \cite{dat1} and the references therein)
\begin{equation}
C_0(t) = \sum_{k=0}^\infty\frac{t^k}{k!^2}.
\end{equation}
that turns out to be related to modified  Bessel function of the first kind and order zero  $I_0$ by
\begin{equation}
C_0(t) = I_0(2t^{1/2})\ .
\end{equation}

In the next section we will also discuss in some detail the relations between the Wright functions of the first kind with functions of the Bessel type.
  
The operator $\displaystyle\frac{d}{dt}t\frac{d}{dt}$ appearing in \eqref{lag} is also named Laguerre derivative in the literature.  
Laguerre derivatives have been recently studied by different authors in the framework 
of the so called monomiality principle pointed out for example in \cite{dat1}.
Applications of Laguerre derivatives in population dynamics have been considered in \cite{bretti}.
More recently, in \cite{zhu}, mathematical models of heat propagation based on Laguerre 
derivatives in space have been studied. 

The aim of this short note is twofold. First of all we provide a more general result connecting Wright functions of the first kind with fractional ODE with variable coefficients. We underline the role of these special functions in the theory of fractional differential equations.
Then, we discuss some simple applications of these results to solve nonlinear
fractional PDEs admitting solutions by generalized separating variable 
solution.

\section{Preliminaries about Wright functions}

The classical \emph{Wright function} that we denote by $W_{\lambda , \mu}(z)$, is defined by the series representation convergent in the whole complex plane,
\begin{equation}
W_{\lambda , \mu}(z) :=\sum_{n=0}^{\infty}{\frac{z^n}{n!\Gamma (\lambda n + \mu)}}, ~~~ \lambda > -1, ~~~ \mu \in \mathbb{C},
\end{equation}

The \emph{integral representation} reads as: 
\begin{equation}
W_{\lambda , \mu}(z) = \frac{1}{2\pi i}\int_{Ha_{-} }{e^{\sigma + z\sigma ^{-\lambda}}\frac{d\sigma}{\sigma ^{\mu}}}, ~~~ \lambda > -1, ~~~ \mu \in \mathbb{C},
\end{equation}
where $Ha_{-}$ denotes the Hankel path: this one is a loop which starts from $-\infty$ along the lower side of negative real axis, encircling it with a small circle the axes origin and ends at $-\infty$ along the upper side of the negative real axis.
 
$W_{\lambda , \mu}(z)$ is then an \emph{entire function} for all
$\lambda \in (-1, +\infty)$.
 Originally, Wright assumed $\lambda \ge 0$ in connection with his investigations on the asymptotic theory of partition
 \cite{Wright 1935} and only in 1940 he considered
  $-1 < \lambda < 0$, \cite{Wright 1940}.

In view of the asymptotic representation in the complex domain 
and of the Laplace transform for positive argument  $z=r>0$
($r$ can be the time variable $t$ or the space variable $x$),
  the Wright functions are distinguished in
   \emph{first kind} ($\lambda \geq 0$) and \emph{second kind}
 ($-1< \lambda < 0$) 
 as outlined in the Appendix F of the book by Mainardi
 \cite{Mainardi BOOK2010}, see also the recent survey article 
 \cite{Mainardi-Consiglio MATHEMATICS2020}.
 In particular, for the asymptotic behavior, we refer the interested reader
to the surveys by Luchko and by Paris in the Handbook of Fractional Calculus
  and Applications,
   see, respectively, \cite{Luchko HFCA,Paris HFCA},
    and references therein. 
 
We note that the Wright functions are an entire of order $1/(1+\lambda)$;
 hence, only  the first kind functions ($\lambda \ge 0$) are of exponential order, whereas  the second kind functions ($-1<\lambda<0$)  
 are not of exponential order.
 The case $\lambda =0$ is trivial since
 $W_{0,\mu}(z) = {\e^z}/{\Gamma(\mu)}.$
As a consequence of the difference in the orders, we must point out the different Laplace transforms proved 
e.g., in  \cite{GOLUMA 99,Mainardi BOOK2010}, see also the 
recent survey
on Wright functions by Luchko \cite{Luchko HFCA}. 
We have:
 \begin{itemize}
\item for the first kind, when $\lambda \ge 0$
\begin{equation}
  W_{\lambda ,\mu } (\pm r) \,\div \,
    \rec{s}\, E_{\lambda ,\mu }\left(\pm \rec{s}\right) \,;
\end{equation}
 
\item for the second kind,   when $-1<\lambda<0$
and putting for convenience  $\nu = -\lambda$
so $0< \nu<1$
\begin{equation}
W_{-\nu  ,\mu } (-r) \,\div \,
    E_{\nu ,\mu+\nu }\left(-s \right) \,,
\end{equation}
 \end{itemize}
whre $E_{\lambda, \mu}(z)$ denotes the Mittag-Leffler function, see for details
\cite{GKMS BOOK2014}.

In the present  paper we need to restrict our attention to the Wright function of the first kind  in order to point out their  relations with functions of the Bessel type.
Indeed, it is easy to recognize 
that the Wright functions of the first kind 
turn out to be related to the well-known
Bessel functions $ J_\nu  $  and $I_\nu$  for $\lambda =1$ and $\mu = \nu+1$. 
In fact, by using the well known series definitions  for the Bessel functions
 and the series definitions  for the Wright functions,  we get the  identities:
\begin{equation}
\begin{array}{ll}
J_{\nu}(z) \!:= \!   {\ds \left(\frac{z}{2}\right )^\nu \,  \sum_{n=0}^{\infty}
\frac{(-1)^n (z/2)^{2n}}{  n!\,  \Gamma( n + +\nu +1)}}  
 \!=\! {\ds \left(\frac{z} {2}\right )^\nu \,
      W_{1, \nu + 1}\left(- \frac{z^2}{4} \right)} ,\\
	W_{1, \nu + 1}\left(-z \right)   
     :=  {\ds  \sum_{n=0}^{\infty} \frac{(-1)^n z^n}{  n!\,  \Gamma( n + \nu +1)}}
	 \!= \!{\ds z^{-\nu /2} \, J_\nu (2 {z}^{1/2})}.  
	 \end{array}
	 \label{eq:F4}
\end{equation}
and 
\begin{equation}
\begin{array}{ll}
   I_{\nu}(z) \!:=\! 
   {\ds \left(\frac{z}{2}\right )^\nu \,  \sum_{n=0}^{\infty}
\frac{(z/2)^{2n}}{  n!\,  \Gamma( n + +\nu +1)}}  
\!=\! {\ds \left(\frac{z}{2}\right )^\nu \,
      W_{1, \nu + 1}\left(\frac{z^2} {4} \right)}\,,\\
   W_{1, \nu + 1}\left(z \right)    
  : =   {\ds  \sum_{n=0}^{\infty}
\frac{z^n}{  n!\,  \Gamma( n + \nu +1)}}
\!=\!  {\ds z^{-\nu /2} \, I_\nu (2 {z}^{1/2})} \,. 
\end{array} 
\label{eq:F5}
\end{equation}
As far as the standard Bessel functions $J_\nu$ are concerned, the following observations 
are worth  noting.
We first note that the Wright functions $W_{1, \nu+1}(-z)$
are related  to the  entire functions ${\mathcal{J}}^C_\nu(z)$
known as {\it Bessel-Clifford  functions} herewith defined 
\begin{equation}
 J_\nu^C  (z) := z^{-\nu /2} \, J_\nu (2 {z}^{1/2}) =
    \sum_{k=0}^{\infty} \frac{(-1)^k z^k}{  k!\,  \Gamma( k + \nu +1)}.
\end{equation}

We note that  different  variants of the Bessel functions (that is  without the singular factor)  were adopted  independently by Tricomi to get  entire functions 
as in the case of Bessel-Clifford  in his treatise on Special Functions published in the late 1950's \cite{Tricomi  BOOK1959}, later revisited and enlarged by Gatteschi \cite{Gatteschi BOOK1973}, 
\begin{equation}
 J_\nu^T(z):=  (z/2)^{-\nu} J_\nu(z)= 
 \sum_{k=0}^\infty \frac{(-1)^k}{k!\Gamma(k +\nu +1)}\, \left(\frac{z}{2}\right)^{2k}\,.
\end{equation}
Then, in view of the first equation in  (\ref{eq:F4}),
some authors refer to the Wright function (of the first kind)  as
the  {\it Wright  generalized  Bessel function}
(misnamed also as the {\it Bessel-Maitland function})
and introduce the notation for $\lambda \ge 0$, see \eg 
\cite{Kiryakova BOOK1994}, p. 336, and \cite{paneva}
\begin{equation}
J_\nu ^{(\lambda)} (z) \!:=\! \left(\frac{z}{2}\right )^\nu 
    \sum_{n=0}^{\infty}
\frac{(-1)^n (z/2)^{2n}}{  n!  \Gamma(\lambda  n + \nu +1)}\! = \!
\left(\frac{z}{  2 }\right )^\nu W_{\lambda, \nu + 1}\left(- \frac{z^2}{4} \right).
 \label{eq:F6}
 \end{equation}
 Similar remarks can be extended to the modified Bessel functions $I_\nu$.
 Even if for Bessel functions the parameter $\nu$ can take aribitrary real and/or complex values, from now on we restrict our analysis to $\nu \ge 0$.
 

\section{Fractional ordinary differential equations with variable coefficients}

We here prove a new connection between Wright functions of first kind and fractional ODE.

\begin{te}
The fractional equation
\begin{equation}
\frac{d^\beta}{dt^\beta}\left(t^{\nu}\frac{df}{dt}\right)= \beta t^{\nu-1} f(t),
\end{equation}
involving a fractional derivative in the sense of Caputo of order $\beta \in (0,1)$, admits a solution of the form
\begin{equation}
f(t) = W_{\beta,\nu}(t^\beta).
\end{equation}
\end{te}

\begin{proof}
By direct calculations we have that
\begin{align}
\nonumber\frac{d^\beta}{dt^\beta}\left(t^{\nu}\frac{d}{dt}\right)W_{\beta,\nu}(t^\beta)& = \beta \frac{d^\beta}{dt^\beta}\sum_{k=1}^\infty \frac{t^{k\beta+\nu-1}}{(k-1)!\Gamma(k\beta+\nu)}\\
\nonumber & =\beta \sum_{k=1}^\infty \frac{t^{k\beta+\nu-1-\beta}}{(k-1)!\Gamma(k\beta+\nu-\beta)} = \beta t^{\nu-1}  W_{\beta,\nu}(t^\beta).
\end{align}
where we used the fact that
\begin{equation}
\frac{d^\beta}{dt^\beta}t^s = \frac{\Gamma(s+1)}{\Gamma(s+1-\beta)}t^{s-\beta},
\end{equation}
where $s>0$.
\end{proof}

This result is clearly more general than the one proved in \cite{fil}, but also in this case we have a direct connection with Bessel-type equations involving Laguerre derivatives. 
Indeed, for $\nu = \beta = 1$ we obtain again the equation
\begin{equation}
\frac{d}{dt}t \frac{d u}{dt} = u,
\end{equation}
studied in \cite{dat1}. See also \cite{zhu} for the applications of Laguerre
derivatives in mathematical models for physics.\\

In a similar way, we can prove that the function
\begin{equation}
u(t) = \sum_{k=0}^\infty  \frac{t^{k\beta}}{k!^2\Gamma(k\beta+\nu)},
\end{equation}
solves the equation
\begin{equation}
\left(\frac{d^\beta}{dt^\beta}t^{\nu}\frac{d}{dt}t\frac{d}{dt}\right)f(t)= \beta t^{\nu-1} f(t).
\end{equation}
This can be simply generalized to higher order equations. 

\section{New results about nonlinear fractional diffusion equations}

A simple but useful method for constructing exact solutions for nonlinear PDEs
is given by the generalized separation of variables. This method permits to find 
particular classes of exact solutions mainly based on the reduction to nonlinear ODEs that
can be exactly solved. It is possible to prove that wide classes of nonlinear PDEs admit
such solutions in separating variable form by using for example the invariant subspace method (see the relevant monograph \cite{gala}). Many papers have been devoted to show the utility of this method in order to obtain exact results for nonlinear diffusive equations, we refer for example to \cite{pol} and \cite{pol1} and the references therein.\\
In the recent literature the construction of exact solutions for nonlinear fractional PDEs based on this method have gained some interest, see for example \cite{saha}. Indeed, there are few
exact results for nonlinear PDEs involving space or time-fractional derivatives. On the other hand, the applications of Lie 
group methods play a central role in this framework for a complete mathematical analysis. We refer in particular 
to the review chapters \cite{kasa} and \cite{kasa1} published on the recent \textit{Handbook of Fractional Calculus with Applications}.\\

We here prove some new interesting results that can be obtained as a direct consequence of Theorem 3.1.

\begin{prop}
The nonlinear fractional equation 
\begin{equation}
t^{1-\nu}\frac{\partial^\beta}{\partial t^\beta}t^{\nu}\frac{\partial u}{\partial t}+u\frac{\partial u}{\partial x}= -u^2-\beta u,\quad \beta \in(0,1),
\end{equation}
admits the solution 
\begin{equation}
u(x,t) = e^{-x}\cdot  W_{\beta,\nu}(t^\beta).
\end{equation}
\end{prop}
\begin{proof}
We search a particular solution by separating variable form $u(x,t) = f(t)e^{-x}$.
By substitution, we have that
\begin{equation}
e^{-x}t^{1-\nu}\frac{\partial^\beta}{\partial t^\beta}\left(t^{\nu}\frac{\partial f}{\partial t}\right)-f^2 e^{-2x}= -f^2e^{-2x}-\beta f e^{-x}
\end{equation}
and therefore to the following fractional ODE on f(t)
\begin{equation}\label{so}
\frac{d^\beta}{d t^\beta}t^{\nu}\frac{d f}{d t}= -\beta t^{\nu-1} f(t). 
\end{equation}
Under the condition that $f(t = 0) = 1$ the solution for \eqref{so} is given by
$$f(t)= W_{\beta, \nu}(-t^\beta)$$
and we obtain the claimed result.
\end{proof}

\begin{os}
Observe that, if we consider the case $\beta= \nu = 1$, we obtain an explicit solution for
the equation 
\begin{equation}
\frac{\partial}{\partial t}t\frac{\partial u}{\partial t}+u\frac{\partial u}{\partial x}= -u^2- u,
\end{equation}
that is a sort of nonlinear hyperbolic equation with an advective term and a nonlinear reaction term, admitting the following particular solution
\begin{equation}
u(x,t) = C_0(-t)\cdot e^{-x}.
\end{equation}
\end{os}

 \begin{prop}
    The equation 
    \begin{equation}\label{nld3}
        t^{1-\lambda}\frac{\partial}{\partial t}t^\lambda\frac{\partial^\lambda u}{\partial t^\lambda} = \frac{\partial^2 u^m}{\partial x^2}-u, \quad m>0,
        \end{equation}
     admits a solution of the form 
     \begin{equation}
     u(x, t) =W_{\lambda,1}(-t^\lambda) \, x^{1/m}.
     \end{equation}
    \end{prop}
     \begin{proof}
     Let us assume that the equation admits a solution, via separation of variables, of the form $x^{1/m}f(t)$, then if we plug in this \textit{ansatz} we get
     \begin{equation}
    \left(t^{1-\lambda}\frac{\partial}{\partial t}t^\lambda\frac{\partial^\lambda}{\partial t^\lambda}\right) x^{1/m} f(t)= -x^{1/m}f(t),
     \end{equation}
     whose solution is given by 
     \begin{equation}
     \nonumber f(t) = W_{\lambda,1}(-t^\lambda),
     \end{equation}
     as claimed.
     \end{proof}

We have previously discussed applications to equations involving fractional Bessel-type operator in time, another interesting case is the applications to linear space-fractional diffusive-type equations involving derivatives in the sense of Caputo. We have the following result.

\begin{prop}
The fractional equation 
\begin{equation}
\frac{\partial u}{\partial t} = \frac{1}{\beta \ x^{\nu-1}}\frac{\partial^\beta}{\partial x^\beta}\left(x^{\nu}\frac{\partial u}{\partial x}\right), \quad x \geq 0, \beta \in(0,1),
\end{equation}
admits a solution of the form
\begin{equation}
u(x,t)= e^{-t}\cdot W_{\beta, \nu}(-x^\beta).
\end{equation}
\end{prop}

We neglect the complete proof that can be directly obtained by substitution.

\section{Conclusions}

The main aim of this paper is to underline a new interesting application of Wright functions of the first kind to solve fractional ordinary differential equations with variable coefficients that generalize Bessel-type equations. In practice the solutions are formal because obtained by the method of substition, showing the direct connection between fractional Bessel-type equations and Wright functions.
Then, we discuss the applications of this new result to solve linear and nonlinear fractional partial differential equations admitting solutions obtained by means of the generalized separation of variable method. As a consequence we are able to find new exact solutions for time or space-fractional linear and nonlinear diffusive-type equations with variable coefficients.\\
A more general classification of linear or nonlinear fractional equations admitting solutions that can be obtained by using the results here discussed should be object of further research. 
Moreover, the possible applications of these fractional Bessel-type equations must be deepened. 

\bigskip

\textbf{Acknowledgments:} The work of the authors has been carried out in the framework of the activities of the National Group
for Mathematical Physics (GNFM).

\end{document}